\documentclass{amsart}
\usepackage{amsfonts}
\usepackage{amsmath}
\usepackage{graphicx,amssymb} 
\usepackage{amsthm}
\usepackage{amssymb}
\usepackage{amscd}

\newtheorem{thm}{Theorem}[section]
\newtheorem{lm}[thm]{Lemma}
\newtheorem{fact}[thm]{Fact}
\newtheorem{cor}[thm]{Corollary}
\newtheorem{prop}[thm]{Proposition}
\newtheorem{exm}[thm]{Example}

\theoremstyle{definition}
\newtheorem*{definition}{Definition}

\newcommand{\wB}{\widetilde{B}}
\newcommand{\wb}{\widetilde{b}}
\newcommand{\ess}{\subseteq_{\text{e}}}

\newcommand{\Z}{\mathbb Z}
\newcommand{\Nn}{\mathbb N}
\newcommand{\Q}{\mathbb Q}
\newcommand{\J}{\mathcal{J}}

\newcommand{\Tc}{\mathcal{T}}
\newcommand{\Uc}{\mathcal{U}}

\newcommand{\End}{\text{\rm End}}

\DeclareMathOperator{\an}{\rm rann}

\DeclareMathOperator{\I}{\rm I}
\DeclareMathOperator{\II}{\rm II}
\DeclareMathOperator{\III}{\rm III}

\begin{document}
\title{Self-injective von Neumann regular rings and K\"othe's Conjecture}
\author[Peter K\'alnai]{Peter K\'alnai}
\email{pkalnai@gmail.com}
\author[J. \v Zemli\v cka ]{Jan \v Zemli\v cka}
\email{zemlicka@karlin.mff.cuni.cz}

\address{Department of Algebra, Charles University in Prague, Faculty of Mathematics and Physics Sokolovsk\' a 83, 186 75 Praha 8, Czech Republic}


\begin{abstract}
One of the many equivalent formulation of the K\"othe's conjecture is the assertion that there exists no ring which contains two nil right ideals whose sum is not nil. We discuss several consequences of an observation that if the Koethe conjecture fails then there exists a counterexample in the form of a countable local subring of a suitable self-injective prime (von Neumann) regular ring.
\end{abstract}

\date{\today}


\renewcommand{\thefootnote}{}

\maketitle
\setcounter{page}{1}


\footnote{2010 \emph{Mathematics Subject Classification}: Primary ; Secondary .} 

In the famous paper \cite{Koth1930}, Gottfried K\"othe asked whether a ring with no non-zero two-sided nil ideal necessarily contains no non-zero one-sided nil ideal. The affirmative answer to this question is usually referred as K\"othe's Conjecture. Despite of plenty established equivalent reformulations and many known particular classes of rings satisfying the conjecture, a general answer on this question is not known. The original formulation and most famous translations of K\"othe's Conjecture use the language of associative rings without unit. Recall several of them \cite{Krem1972,Smok2005,FerrPucz1989} (see also \cite{Smok2000,Smok2001,Smok2006}):
\begin{itemize} 
\item  the matrix ring $M_n(R)$ is nil for every nil ring $R$,
\item $R[x]$ is Jacobson radical for every nil ring $R$,
\item $R[x]$ is not left primitive for every nil ring $R$,
\item every ring which is a sum of a nilpotent subring and a nil subring is nil. 
\end{itemize}

The present paper deals with its characterization in terms of unital rings, namely, we say that a (unital associative) ring \emph{satisfies the condition (NK)} if it contains two nil right ideals whose sum is not nil. K\"othe's Conjecture is then equivalent to the property that there exists no ring satisfying the condition (NK) \cite[10.28]{Lam2001}. Recall that the conjecture holds for all right Noetherian rings \cite[10.30]{Lam2001}, PI-rings \cite{Brau1984} and rings with right Krull dimension \cite{Lena1973}.

The main goal of this text is to translate properties of a potential counterexample for K\"othe's Conjecture to a certain structural question about simple self-injective Von Neumann regular rings. Our main result, Theorem~\ref{PhD3_MainResult_IIf_III}, shows that existence of a ring satisfying (NK) implies existence of a countable local ring satisfying (NK) which is a subring of a self-injective simple VNR ring either Type $\II_f$ or Type $\III$. 

As it was said, all rings in this paper are supposed to be associative with unit. An ideal means a two-sided ideal and $C$-algebra is any ring $R$ with a subring $C$ contained in the center of $R$. A (right) ideal $I$ is called \emph{nil} whenever all elements $a\in I$ are nilpotent, i.e. there exists $n$ such that $a^n=0$. A ring $R$ is said to be {\it Von Neumann regular} (VNR) if for every $x\in R$ there exists $y\in R$ such that $x=xyx$. A VNR ring is called {\it abelian regular} provided all its idempotents are central.
the Jacobson ideal of a ring $R$ will be denoted by $J(R)$.

For non-explain terminology we refer to \cite{Good1979_3} and \cite{Sten1975}.

\section{Algebras over $\Q$ and $\Z_p$}

Let $R$ be a ring, $a\in R$, and $I$ an ideal. The element $a$  is said to be \emph{nilpotent modulo} $I$ provided there exists $n$ such that $a^n\in I$, and $a$ is \emph{principally nilpotent} provided that the right ideal $aR$ is nil. Note that then also the left ideal generated by $a$ is nil, $a$ is an element of the Jacobson ideal and each $\alpha a \beta \in R$ is principally nilpotent as well for every $\alpha, \beta \in R$. We say that $a$ is a \emph{minimal non-nilpotent} element of $R$ if $a$ is not nilpotent and $a$ is nilpotent modulo $J$ for every non-zero ideal $J$. It is easy to see that a ring satisfies (NK) if there exist two principally nilpotent elements whose sum is not nilpotent.

\begin{exm}\label{commutative example}
For every non-nilpotent element $a$ of a commutative ring $R$ there exists a prime ideal $P$, which is exactly a maximal ideal disjoint with the set $\{a^n\mid n\in\mathbb N\}$, such that $a+P$ is a minimal non-nilpotent element of the commutative domain $R/P$.
\end{exm}

First, we make two elementary observations on generators of algebras satisfying (NK) and on the existence of minimal non-nilpotent elements:

\begin{lm}\label{PhD3_ReductionStep1} 
	Let $R$ be a $C$-algebra satisfying the condition (NK). Then there exists a $C$-subalgebra $S$ of $R$ generated by two elements $x,y$ such that both $x$ and $y$ are principally nilpotent in $S$, but $x+y$ is not nilpotent. 
\end{lm}
\begin{proof} 
Since there exist two nil right ideals $K$ and $L$ of $R$ such that $K+L$ is not nil, there exists elements $x \in K$ and $y\in L$ such that $x+y$ is not nilpotent. If $S$ denotes a $C$-subalgebra of $R$ generated by $\{x,y\}$, then $xS$ and $yS$ are nil right ideals, as $xS\subseteq xR$ and  $yS\subseteq yR$. 
\end{proof}

The classical commutative construction of Example~\ref{commutative example} can be easily generalized in non-commutative setting. 

\begin{lm}\label{PhD3_MinimalNonnilpotentPrimeFactor} 
	Let $R$ be a ring and $a\in R$. If $a$ is not nilpotent, then there exists a prime ideal $I$ such that $a+I$ is a minimal non-nilpotent element of the factor ring $R/I$. 
\end{lm}
\begin{proof} 
Let us take a maximal ideal $I$ that does not contain any power of $a$, which exists by Zorn's Lemma. Obviously $a$ is not nilpotent modulo $I$ and it is nilpotent modulo $J$ for every ideal $J\supset I$. Since for any ideals $U$ and $V$ such that $I\subset U,V$ 
there exist $m,n \in \Nn$ satisfying $a^m\in U$ and $a^n\in V$, hence $a^{m+n}\in UV$. Now, it is clear that $UV \not \subseteq I$, thus $I$ is a prime ideal.
\end{proof}

As the consequence we can easily see that, if the class of all algebras over a field satisfying the (NK) condition is non-empty, then we can choose countably-dimensional one: 

\begin{cor} 
If $F$ is a field and $R$ is an $F$-algebra which satisfies (NK), then there exists an $F$-subalgebra $A$ of $R$ satisfying (NK) such that $\dim_F(A) \le \omega$.
\end{cor}

Recall that the fact that the Jacobson radical of a countably-dimensional algebra over an uncountable field is nil implies that there is no $F$-algebra satisfying (NK) over uncountable field $F$  \cite[Corollary 4]{Amit1956}. It implies that for search of algebras over a field satisfying (NK) we can restrict our attention just to countable fields and to algebras of countable cardinality. Nevertheless, we will show below that it is enough to research existence of general rings satisfying (NK) in the class of countable generated algebras either over the field of rational numbers or over field $\Z_p$ for a prime number $p$.

\begin{lm}\label{PhD3_ReductionStep2} 
 Let $R$ be a $C$-algebra generated by two principally nilpotent elements $x,y$ such that $x+y$ is not nilpotent. Then there exits a prime ring $S$ satisfying (NK) and a surjective ring homomorphism $\pi: R \to S$ such that $\pi(C)$ is a commutative domain, $S$ is a $\pi(C)$-algebra generated by principally nilpotent elements $\pi(x), \pi(y)$ contained in $\J(S)$, and $\pi(x)+\pi(y)$ is a minimal non-nilpotent element.
\end{lm}
\begin{proof}
Let $I$ be a prime ideal such that $x+y+I$ is a minimal non-nilpotent element of the factor ring $R/I$ which exists by Lemma~\ref{PhD3_MinimalNonnilpotentPrimeFactor}. Put $S=R/I$ and denote by $\pi$ the canonical projection $R \to S$. Then $\pi(C) \cong C/C\cap I$ is a commutative domain and $\pi(x), \pi(y)$ generates $S$ as a $\pi(C)$-algebra. Obviously, $\pi(x)+\pi(y)=x+y+I$ is a minimal non-nilpotent by the construction and homomorphic images of nil right ideals $\pi(xR)=\pi(x)S$, $\pi(yR)=\pi(y)S$ are nil.
\end{proof}

Given a ring $R$, we denote by $C_R$ the subring generated by the unit of $R$. Since $C_R$ is isomorphic either to $\mathbb Z$ or to $\mathbb Z_n$ for an integer $n \in \Nn$ and $R$ has the natural structure of a $C_R$-algebra, we obtain the following immediate consequence of the Lemmas~\ref{PhD3_ReductionStep1}, \ref{PhD3_ReductionStep2}:

\begin{cor} 
If there exists a ring satisfying (NK), then there exists an algebra over either $\Z$ or $\Z_p$ for some prime number $p \in \Nn$ generated by two principally nilpotent elements which satisfies (NK).
\end{cor}

Before we formulate claim that we can deal with $\Q$-algebras instead of $\Z$-algebras we make the following easy observation:

\begin{lm}\label{PhD3_2PrincNilpotentGeneratedFieldAlgebraNKLocal} 
Let $F$ be a field and $R$ be an $F$-algebra generated by two principally nilpotent elements $x,y$ such that $x+y$ is not nilpotent. 
Then $R$ is a local ring satisfying (NK) with $\J(R)=xR+yR=Rx+Ry$ .
\end{lm}
\begin{proof}
Since $x$ and $y$ are principally nilpotent, $xR+yR\subseteq \J(R)$ and 
$Rx+Ry\subseteq \J(R)$. Observe that $R=F+xR+yR=F+Rx+Ry$ which implies $R/(xR+yR)\cong F$, hence $\J(R)\subseteq xR+yR$ and $\J(R)\subseteq Rx+Ry$.
\end{proof}

\begin{prop}\label{PhD3_NKReductionFinal} 
Let $R$ be a ring satisfying (NK). Then there exists a subring $S$ of $R$ generated by two elements $\xi$, $\upsilon$, a prime $F$-algebra $A$ and an epimorphism of rings $\varphi: S \to A$ such that the following conditions hold for elements $x=\varphi(\xi)$ and 
$y=\varphi(\upsilon)$:
\begin{enumerate}
\item[(K1)] either $F=\Q$ or $F=\Z_p$ where $p$ is a prime number,
\item[(K2)] $x$ and $y$ are principally nilpotent generators of $A$, and $x+y$ is a minimal non-nilpotent element,
\item[(K3)] $x A + y A=\J(A)$ is the unique maximal right (left) ideal of $A$.
\end{enumerate}
The $F$-algebra $A$ satisfies (NK) and if $R$ is a PI-algebra then $A$ can be taken as a PI-algebra.
\end{prop}
\begin{proof} Since $R$ has a structure of $C_R$-algebra, we can take a subring $S$ generated by principally nilpotent elements $\xi$, $\upsilon$ such that $\xi+\upsilon$ is not nilpotent which is ensured by Lemma~\ref{PhD3_ReductionStep1}. Moreover, $C_R=C_S$ and   we may suppose that $\xi+\upsilon$ is a minimal non-nilpotent element and $C_S\cong \Z$ or $C_S\cong \Z_p$ by Lemma~\ref{PhD3_2PrincNilpotentGeneratedFieldAlgebraNKLocal}. For each integer $p\in \Nn$ let $\mu_p: S \to S$ be the $S$-endomorphism induced by multiplication by $p$ and note that $\ker(\mu_p)$ is an ideal for an arbitrary $p$.

First, suppose that there exists a prime number $p$ such that $\ker (\mu_p)\ne 0$. Since $\xi+\upsilon$ is minimal non-nilpotent, it follows by the hypothesis that there exists $n\in \Nn$ for which $(\xi+\upsilon)^n\in\ker(\mu_p)$, i.e. $p(\xi+\upsilon)^n=0$. Put 
\begin{equation*}
 x'=\xi(\xi+\upsilon)^n,  y'=\upsilon(\xi+\upsilon)^n\in\ker(\mu_p),
\end{equation*}
and denote by $S'$ a $C_S$-subalgebra of $S$ generated by $\{x',y' \}$. Since $px'=0=py'$ and $p$ is not an invertible element, we get that $pS'= pC_S \ne C_S$, hence $C_S/pC_S \cong \Z_p$ and we may identify $C_S/pC_S$ and $\Z_p$. Furthermore, assume that $(x'+y')^k\in pS'$ for some $k$. Then there exists $c\in C_S$ such that $(x'+y')^k=pc$ because $pS'= pC_S$, and so 
\begin{equation*}
0\ne (\xi+\upsilon)^{2k(n+1)} = (x'+y')^{2k}=(pc)^2=p(x'+y')^kc=0,
\end{equation*}
a contradiction. Hence $(x'+y')^k\notin pS'$ for all $k$. Note that $S'/pS'$ is a 
$\Z_p$-algebra (in fact $C_S/pC_S$-algebra) generated by principally nilpotent elements $x'+pS'$ and $y'+pS'$ with non-nilpotent $x'+y'+pS'$. 
So there exists a prime factor of $S'/pS'$, denote it by $A$, which is a $\Z_p$-algebra satisfying (K2) by Lemma~\ref{PhD3_ReductionStep2}. Obviously, $S'/pS'$ satisfies also (K3) by Lemma~\ref{PhD3_2PrincNilpotentGeneratedFieldAlgebraNKLocal}.
Because the naturally constructed homomorphism $\varphi: S \to A$ is a composition of surjective homomorphisms, it is an epimorphism.

Now, suppose that $\ker(\mu_p)= 0$ for all primes $p$, hence $S_\Z$ is a torsion-free abelian group. Denote by $E(S_\Z)$ the injective envelope of $S_\Z$ as a $\Z$-module. Since every endomorphism of $S_\Z$ can be extended to an endomorphism of $E(S_\Z)$ and  multiplication by each element of $S$ can be viewed as an endomorphism of $S_\Z$, there exist $x,y \in \End_\Z(E(S_\Z))$ such that $x(s)=\xi \cdot s$ and $y(s)=\upsilon\cdot s$ for every $s \in S$. As $S_\Z$ is torsion-free, $\End_\Z(E(S_\Z))$ has a natural structure of a $\Q$-algebra. Let us denote by $\overline{A}$ its $\Z$-subalgebra generated by $x$ and $y$, by $A$ its $\Q$-subalgebra generated by $x$ and $y$ and by $i:\overline{A}\to A$ the inclusion homomorphism. Note that a map $\psi:\overline{A}\to S$ defined by the rule $\psi(\alpha)=\alpha(1)$ provides correctly defined isomorphism of $\Z$-algebras such that $\psi(x)=\xi$ and $\psi(y)=\upsilon$. Furthermore, $A$ can be taken as prime by Lemma~\ref{PhD3_ReductionStep2} and it is easy to see that $i \circ \psi^{-1}:S\to A$ is a ring epimorphism.

Finally, if $r\in A$, there exists $m\in \mathbb Z\setminus\{0\}$ for which $ rm \in \overline{A}$ which implies that there is $k \in \Nn$ such that $m^k(xr)^k=(xrm)^k=0$. As $\End_\Z(E(S_\Z))$ is a torsion-free abelian group, $(xr)^k=0$ which shows that $x$ is a principal nilpotent element. The same argument applied on $y$ proves that $y$ is principal nilpotent as well. Now $A$ is $\Q$-algebra satisfying (K2) and (K3) again by Lemmas~\ref{PhD3_ReductionStep2} and \ref{PhD3_2PrincNilpotentGeneratedFieldAlgebraNKLocal}.

To prove the addendum, suppose that $R$ is a PI-ring and note that the class of PI-algebras is closed under taking subrings and factor rings. Thus $F$-algebras $A$ are PI-algebras for finite fields $F$. If $F=\Q$ we can see that $A\cong \Q \otimes \overline{A}$ which is polynomial by \cite[Theorem 6.1]{Rowe1980}.
\end{proof}

The previous claim easily allows to restrict reformulation of K\"othe's Conjecture due to Krempa \cite[Theorem 6]{Krem1972} just to fields $\Q$ and $\Z_p$:

\begin{cor} \label{PhD3_NKReductionFinal_Cor}
	K\"othe's Conjecture holds if and only if there is no countable generated $F$-algebra satisfying (NK) for either $F=\Q$ or $F=\Z_p$ where $p$ is a prime number.
\end{cor}

Applying an old Amitsur's result we can reprove a well-known fact that PI-rings satisfies K\"othe's Conjecture. Let us first state a general result by Braun extending previous works in \cite{Raz1974, Kem1980}.
 
\begin{fact}\cite[Theorem 5]{Brau1984}\label{PhD3_JacPINoeCommRingNilpotent} 
	The Jacobson radical of a finitely generated PI-algebra over a Noetherian commutative ring is nilpotent.
\end{fact}

\begin{prop}\label{PhD3_NoPINK}  
	There is no PI-ring satisfying (NK).
\end{prop}
\begin{proof} 
Let us assume that $R$ is a PI-ring satisfying (NK). Then by Proposition~\ref{PhD3_NKReductionFinal} there exists a 2-generated PI-algebra $A$ over a field $F$ with the Jacobson radical of $A$ not nil, hence we get a contradiction with Fact~\ref{PhD3_JacPINoeCommRingNilpotent}.
\end{proof}

Since an algebra over a field is easily embeddable into a VNR ring, we are able to find a prime VNR extension of algebras constructed in Proposition~\ref{PhD3_NKReductionFinal}.

\begin{lm}\label{PhD3_NKEmbedCountPrimeVNR} 
	An algebra $A$ satisfying the conditions (K1)--(K3) from Proposition~\ref{PhD3_NKReductionFinal} is embeddable into a countable prime VNR ring $R$ such that $J\cap A\ne 0$ for each nonzero ideal $J$ of $R$.
\end{lm}
\begin{proof} 
As $A$ is an algebra over a field $F$, there exists a canonical embedding $\nu$ of $A$ into the endomorphism ring $\End_F(A)$. Note that $\End_F(A)$ is a VNR $F$-algebra and for each $x\in\End_F(A)$ fix an element $y_x$ such that $x y_x x=x$. Now we will construct by induction a chain of subrings $R'_i\subset R'_{i+1}\subset\dots\subset \End_F(A)$. We put $R'_1 = \nu(A)$ and if $R'_i$ is defined, $R'_{i+1}$ is an $F$-subalgebra of $\End_F(A)$ generated by the set $R'_{i} \cup \{ y_x \mid x\in R'_{i} \}$. Then $R' = \bigcup_{i \in \Nn} R'_i$ is a countable VNR $F$-algebra containing $\nu(A) \simeq A$ as an $F$-subalgebra.

Fix a maximal ideal $M$ of $R'$ such that $M \cap A=0$. Then the map $a \to a+M, a \in A$ induces an embedding of $A$ into a countable VNR ring $R'/M$. We may identify $A$ with the image of the embedding. Note that if $\overline{I},\overline{J}$ are two nonzero ideals of $R'/M$, then the intersection $\overline{I}\cap A$ resp. $\overline{J}\cap A$ forms a nonzero ideal of $A$. Since $A$ is prime, $0 \ne (\overline{I}\cap A)(\overline{J}\cap A)\subset \overline{I} \overline{J}$, hence $R:=R'/M$ is prime too.
\end{proof}

We finish the section with an important observation about nil ideals due to  Herstein and Small \cite{HerSma1964}:

\begin{lm}\label{PhD3_ACCOnRannExtRing} 
If a ring $Q$ satisfies (ACC) on right and left annihilators and $R$ is a subring of $Q$, then every nil right ideal of $R$ is nilpotent.
\end{lm}
\begin{proof} 
Since any nil right ideal of $R$ forms a nil subring of $Q$ as a ring without unit, the result
follows immediately from  in \cite[Theorem 1]{HerSma1964}.
\end{proof}

\section{Self-injective VNR rings}

Before we present a construction of simple self-injective VNR rings containing rings satisfying (NK), we need to recall several notions and structural results concerning self-injective VNR rings. 

Let $R$ be a VNR ring. An idempotent $e$ is called \emph{abelian} if the ring $eRe$ is abelian regular and $e$ is \emph{directly finite} if the ring $eRe$ is directly finite \cite[p.110]{Good1979_3}, i.e. $xy = 1$ implies $yx = 1$ for all $x,y \in eRe$. A self-injective VNR ring $R$ is \emph{purely infinite} if it contains no nonzero directly finite central idempotent and it is
\begin{itemize}
\item \emph{Type I} if every nonzero right ideal contains a nonzero abelian idempotent \cite[10.4]{Good1979_3},
\item \emph{Type II} if every nonzero right ideal contains a nonzero directly finite idempotent \cite[10.8]{Good1979_3},
\item \emph{Type III} if it contains no nonzero directly finite idempotent.
\end{itemize}
Moreover, $R$ is Type $\I_f$ (Type $\II_f$) provided it is Type $\I$ (Type $\II$) and directly finite. Next, $R$ is Type $\I_\infty$ (Type $\II_\infty$) if it is Type $\I$ (Type $\II$) and purely infinite. Recall the structural decomposition of self-injective VNR rings: 

\begin{fact}\cite[Theorem 10.22]{Good1979_3}\label{PhD3_SelfInjVNRDecomposition}
Every self-injective VNR ring is uniquely a direct product of rings Type $\Tc=\I_f$, $\I_\infty$, $\II_f$, $\II_\infty$, and $\III$. In particular, each prime self-injective VNR ring is exactly one type from $\Tc$.
\end{fact}  

Let $R$ be a right self-injective VNR ring. Define a function $\mu$ on non-singular injective right $R$-modules as follows: for an injective non-singular module $M$, if there exists a non-zero central idempotent $e$ such that $Me = 0$, then $\mu(M) = 0$. Otherwise set $\mu(M)$ to be the smallest infinite cardinal $\alpha$ such that $Me$ does not contain a direct sum of $\alpha$ nonzero pairwise isomorphic submodules for some non-zero central idempotent $e$ \cite[p.143]{Good1979_3}. If $R$ is moreover prime and $\alpha$ is a cardinal, then we set
\begin{equation*}
H(\alpha):= \{ r \in R \mid \mu(rR) \leq \alpha\}
\end{equation*}


\begin{fact}\label{PhD3_LatticeIdealsSelfinjVNR}\cite[Corollary 12.22]{Good1979_3} 
Let $R$ be a prime, right self-injective VNR ring either Type I or Type II. Then the rule $\alpha\to H(\alpha)$ defines a lattice isomorphism between the lattice of non-zero two-sided ideals 
and the cardinal interval  $[\omega, \mu(R)]$.
\end{fact}

\begin{thm}\label{PhD3_MainResult_IIf_III}  
	If there is a ring satisfying (NK) then there exists a countable local subring of a suitable self-injective simple VNR ring of type either $\II_f$ or $\III$ that also satisfies (NK). 
\end{thm}
\begin{proof} 

Let $A$ be a prime $F$-algebra over a field $F$ with two principally nilpotent generators $x$ and $y$ satisfying the conditions (K1)--(K3), which exists by Proposition~\ref{PhD3_NKReductionFinal}, let us denote by $x=\varphi(\xi)$ and $y=\varphi(\upsilon)$ are images of the constructed epimorphism $\varphi$  from the proposition where $x+y$ is a minimal non-nilpotent element. Let $R$ be a countable prime VNR ring extension of $A$ ensured by Lemma~\ref{PhD3_NKEmbedCountPrimeVNR}. Denote by $Q:=Q_{max}(R)$ the maximal right ring of quotients of $R$. Let $I, J\subseteq Q$ be two non-zero ideals of $Q$. Since $R$ is essential in $Q_R$ by \cite[Corollary 2.3]{Sten1975} and both $I$, $J$ are right $R$-submodules of $Q$, then $0\ne (I \cap R)(J \cap R)\subseteq IJ$ by Lemma~\ref{PhD3_NKEmbedCountPrimeVNR}, hence $Q$ is prime. By Fact~\ref{PhD3_SelfInjVNRDecomposition} we get that $Q$ is exactly one of type from $\Tc$.

Assume that $Q$ is Type $\I_f$. Then it is (right and left) Artinian by \cite[Corollary 10.3]{Good1979_3}, which implies that $Q$ satisfies (ACC) on right and left annihilators. 
Hence the nil ideal $J(A)$ of $A$ is nilpotent, a contradiction.

Suppose that $Q$ is of type $\I_\infty$ or $\II_\infty$. Note that $Q$ does not contain any uncountable direct sum of right ideals since $Q_{R}$ is an essential extension of a countable ring $R$, i.e. $\mu(Q_Q)\le\omega_1$. Moreover, it is purely infinite, so by \cite[Theorem 10.16]{Good1979_3}(a)$\to$(d), $Q_Q$ contains a countably generated free submodule which yields $\mu(Q_Q) >\omega$, thus $\mu(Q_Q) = \omega_1$. By Fact~\ref{PhD3_LatticeIdealsSelfinjVNR} (cf. also \cite[Theorem 7.3]{Busq92}), $Q$ contains exactly one nontrivial ideal $H(\omega)$. As  $I=H(\omega)\cap A$ is a nonzero ideal of $A$ by Lemma~\ref{PhD3_NKEmbedCountPrimeVNR} and $x+y$ is a minimal non-nilpotent element of $A$ by the condition (K2) of
 Proposition~\ref{PhD3_NKReductionFinal}, there exists $n\in \Nn$ such that $(x+y)^n\in I$. Also $(x+y)^nx, (x+y)^ny \in (x+y)^nQ = eQ \cap I$ for an idempotent $e \in H(\omega)$. Put $\xi:=(x+y)^nx$, $\upsilon:=(x+y)^ny$ and let $B$ denote the $F$-subalgebra of $eQe$ generated by elements $\xi e$ and $\upsilon e$ and $\wB$ denote the $F$-subalgebra of $Q$ generated by elements $\xi$ and $\upsilon$. We claim that the elements $\xi e$ and $\upsilon e$ are principally nilpotent in $B$ and that $\xi e+\upsilon e$ is not nilpotent, hence $B$ satisfies (NK). 

Indeed, since $\wB \subseteq eQ+F$, we can see that $\wB=e\wB+F$ and so $B=\wB e=e \wB e$. Hence every element from the right ideal $(\xi e)B$ of $B$ has the form $\xi \wb e$ for some element $\wb \in \wB$. Because $\xi$ and $\upsilon$ are principally nilpotent in $A$, so in $\wB$ as well, there exists $k \in \Nn$ such that $\left(\xi \wb \right)^k=0$. As $e \xi e=\xi e$ we get $\left(\xi\wb e\right)^k=\left(\xi \wb \right)^k e=0$. This proves that $\xi e$ is a principally nilpotent element.
The argument for $\upsilon e$ is the same. Finally, $(\xi e+\upsilon e)^k=(x+y)^{k(n+1)}e^k\ne 0$ because $(x+y)^{k(n+1)}e^k(x+y)=(x+y)^{k(n+1)+1}\ne 0$.

Since $zQ$ is a direct summand of an injective module $Q_Q$, it is directly finite for any $z\in H(\omega)$ by \cite[Proposition 5.7]{Good1979_3}. Thus $eQe \simeq End_Q(eQ)$ is directly finite by \cite[Lemma 5.1]{Good1979_3} and right self-injective by \cite[Corollary 9.3]{Good1979_3}. Clearly, $eQe$ is prime and it contains the subalgebra $B$ which is local and satisfies the condition (NK) by Lemma~\ref{PhD3_2PrincNilpotentGeneratedFieldAlgebraNKLocal}.

Recall from the initial part of the proof that $eQe$ can not be Type $\I_f$. So we have proved that there exists a prime right self-injective VNR ring $eQe$ being Type either $\II_f$ or $\III$ which contains a local ring satisfying (NK). Finally note that if $Q$ is Type $\II_f$, then it is simple by \cite[Proposition 9.26]{Good1979_3} and, if $Q$ is Type $\III$ then it is simple by \cite[Theorem 12.21]{Good1979_3}.
\end{proof}

Observe that a non-Artinian self-injective simple VNR ring either Type $\II_f$ or Type $\III$ is necessarily uncountable because it contains an infinite set of orthogonal idempotents.


We conclude the paper by examples of self-injective simple VNR rings Type $\II_f$ or $\III$:
\begin{exm}
Let $F$ be a field  and $\Uc$ be a non-principal ultrafilter on $\Nn$. Put $R=\prod_{n \in \Nn} M_n(F)$ and $I=\{r\in R \mid (\exists U\in \Uc)(\forall i\in U)\pi_i(r)=0 \}$ where $\pi_i:R\to M_i(F)$  denotes the natural projection. Then $I$ is a maximal ideal of $R$, hence $R/I$ forms a simple self-injective VNR ring Type $\II_f$ by \cite[Theorem 10.27]{Good1979_3}.
\end{exm}

\begin{exm}\cite[Example 10.11]{Good1979_3}
Let $F$ be a field, $Q=\End_F(F^{(\omega)})$ and $M=\{x\in Q\mid \dim_F(xF^{(\omega)})<\omega\}$. Then $Q/M$ is a simple VNR ring.
If $R$ is the maximal right ring of quotients of $Q/M$ then $R$ is a simple self-injective VNR ring Type $\III$.
\end{exm}

\end{document}